\documentclass[a4, 12pt]{amsart}
\usepackage[mathscr]{eucal}
\usepackage{amssymb}
\usepackage{latexsym}
\usepackage{amsthm}
\usepackage{color}
\theoremstyle{plain}
\newtheorem{theorem}{Theorem}[section]

\newtheorem{remark}{Remark}[section]
\newtheorem{lemma}{Lemma}[section]

\makeatletter
\@addtoreset{equation}{section}

\title[Complete $\lambda$-surfaces in $\mathbb R^3$]
{Complete $\lambda$-surfaces in $\mathbb R^3$}
\author [Q. -M. Cheng and G. Wei]{Qing-Ming Cheng  and Guoxin Wei}
\address{Qing-Ming Cheng \\  \newline \indent Department of Applied Mathematics, Faculty of Sciences,
\newline \indent Fukuoka University, Fukuoka  814-0180, Japan.  \newline \indent cheng@fukuoka-u.ac.jp}
\address{Guoxin Wei \\  School of Mathematical Sciences, South China Normal University,
\newline \indent 510631, Guangzhou,  China. \newline \indent  weiguoxin@tsinghua.org.cn}

\begin{document}
\maketitle

\begin{abstract}
The purpose of this paper is to study complete $\lambda$-surfaces in Euclidean space $\mathbb R^3$.
A  complete classification
for 2-dimensional complete   $\lambda$-surfaces in Euclidean space $\mathbb R^3$  with constant squared norm of the
second fundamental form is given.

\end{abstract}


\footnotetext{The first author was partially  supported by JSPS Grant-in-Aid for Scientific Research (B):  No.16H03937.
The second author was partly supported by grant No. 11771154 of NSFC.}

\section{introduction}
\vskip2mm
\noindent
\noindent
One of the most important problems in   mean curvature flow is to understand
the possible singularities that the flow goes through.  A key starting point
for singularity analysis is Huisken's monotonicity formula. The monotonicity
implies that the solution to the flow is asymptotically self-similar near a given type I singularity. Thus, it
 is modeled  by  self-shrinking solutions of the flow.
An  $n$-dimensional  submanifold  $X: M\rightarrow \mathbb{R}^{n+p}$  in the $(n+p)$-dimensional
Euclidean space $\mathbb{R}^{n+p}$  is called a self-shrinker if it satisfies
\begin{equation*}
\vec H+ X^{\perp}=0,
\end{equation*}
where  $X^{\perp}$  and $\vec H$ denote the normal part of the position vector $X$ and mean curvature vector of  this
submanifold.
It is known that self-shrinkers play an important role in the study on singularities  of the mean curvature flow because
they describe all possible  blow-ups at a given singularity.   \newline
For the classification of complete self-shrinkers with co-dimension $1$,  many nice works were done. Abresch and Langer \cite{AL}  classified
closed self-shrinkering curves in $\mathbb{R}^2$ and showed that the  round circle is the only embedded self-shrinker.
Huisken \cite{H2, H3},  Colding and Minicozzi \cite{CM}  classified $n$-dimensional  complete  embedded self-shrinkers
in $\mathbb{R}^{n+1}$ with  mean curvature $H\geq 0$ and  with polynomial volume growth.
According to the results  of Halldorsson \cite{H},
Ding and Xin \cite{DX1}, Cheng and Zhou \cite{CZ}, one knows
that for any positive integer $n$,
$\Gamma \times \mathbb R^{n-1}$ is  a complete self-shrinker without polynomial volume growth in $\mathbb R^{n+1}$, where
$\Gamma$ is a complete self-shrinking curve of Halldorsson  \cite{H}. Hence, the  condition of polynomial volume growth in \cite{H3} and \cite{CM}
is essential. Furthermore,  for the study on the rigidity of complete self-shrinkers, many important works have been done
(cf. \cite{CL},   \cite{CHW},  \cite{CO},  \cite{CP}, \cite{CW}, \cite{DX1}, \cite{DX2}, \cite{LW}, \cite{LW1}, \cite{LW2}  and so on). In particular, by estimating the first eigenvalue of the Dirichlet  eigenvalue problem,
Ding and Xin \cite{DX2}  studied 2-dimensional complete self-shrinkers with polynomial volume growth. They  proved
that a $2$-dimensional complete self-shrinker  $X: M\rightarrow \mathbb{R}^{3}$  with polynomial volume growth
and with constant  squared norm $S$ of the second fundamental form
is isometric to one of
 $\mathbb{R}^{2}$,
 $S^1 (1)\times \mathbb{R}$ and $S^{2}(\sqrt{2})$.
Cheng and Peng in \cite{CP}  have proved that for an $n$-dimensional complete self-shrinker
 $X:M^n\rightarrow \mathbb{R}^{n+1} $  with   $\inf H^2>0$,
if  the squared norm $S$ of the second fundamental form is constant, then $M^n$ is  isometric to either
$S^n(\sqrt{n})$ or $S^m(\sqrt m)\times\mathbb{R}^{n-m}$ in $\mathbb{R}^{n+1}$, $1\leq m\leq n-1$.

\noindent Recently,  Cheng and Ogata \cite{CO} have given a complete classification for 2-dimensional complete self-shrinkers with
 constant  squared norm $S$ of the second fundamental form, that is, they have proved the following:

\vskip3mm
\noindent
{\bf Theorem CO.}
{\it A $2$-dimensional complete self-shrinker  $X: M\rightarrow \mathbb{R}^{3}$  in $\mathbb{R}^{3}$ with constant  squared norm of the second fundamental form
is isometric to one of the following:
\begin{enumerate}
\item $\mathbb{R}^{2}$,
\item
 $S^1 (1)\times \mathbb{R}$,
\item  $S^{2}(\sqrt{2})$.
\end{enumerate}
 }

\vskip3mm
\noindent
On the other hand, from a view of variations,
self-shrinkers of mean curvature flow can be characterized as  critical points of the weighted area functional. In \cite{CW2},
the authors  gave  a definition of weighted volume and studied  the weighted area functional for variations preserving this volume.
Critical points for
the weighted area functional for variations preserving this volume  are called  $\lambda$-hypersurfaces  by the authors in \cite{CW2}.
Precisely,   an $n$-dimensional hypersurface  $X:M\to  \mathbb{R}^{n+1}$ in Euclidean space  $ \mathbb{R}^{n+1}$ is  called a $\lambda$-hypersurface  if
\begin{equation}
\langle X, N\rangle +H=\lambda,
\end{equation}
where $\lambda$ is a constant,  $H$ and $N$ denote the mean curvature and  unit normal vector of $X:M\to  \mathbb{R}^{n+1}$, respectively.

\begin{remark}
If $\lambda=0$, $\langle X, N\rangle +H=\lambda=0$, then $X:M\to  \mathbb{R}^{n+1}$ is a self-shrinkers. Hence,  the notation of
$\lambda$-hypersurfaces is a natural generalization of the self-shrinkers of the mean curvature flow.
\end{remark}

\vskip2mm
\noindent
It is well-known that there are three standard examples of $\lambda$-hypersurfaces in $\mathbb{R}^{n+1}$: the $n$-dimensional Euclidean space $\mathbb{R}^{n}$, the $n$-dimensional sphere $S^n(r)$ and the $n$-dimensional cylinder $S^k(r)\times \mathbb{R}^{n-k}$. In \cite{CW4}, Cheng and Wei have constructed compact rotational symmetric   $\lambda$-hypersurfaces.
Very recently, Ross \cite{R}, Li  and the second author \cite{LWe}  have constructed very interesting compact $\lambda$-hypersurfaces.
For recent years, the study on $\lambda$-hypersurfaces has attracted a lot of attention. For example,
in \cite{CW2}, Cheng and Wei have proved
that $S^k(r)\times \mathbb{R}^{n-k}$, $0\leq k\leq n$, are the only  complete embedded $\lambda$-hypersurfaces
with polynomial area growth in $\mathbb{R}^{n+1}$ if  $H-\lambda\geq 0$ and $\lambda(f_3(H-\lambda)-S)\geq0$,
where $f_3=\sum_{j=1}^n\lambda_j^3$, $\lambda_j$ is the principal curvature of this hypersurface. In \cite{CW3}, the authors have
studied the growth on upper and lower bounds of area for complete $\lambda$-hypersurfaces and stability for compact $\lambda$-hypersurfaces.
A number of nice works are about the study of rigidity for complete
$\lambda$-hypersurfaces under point-wise pinching conditions or global pinching conditions (\cite{COW}, \cite{CW4}, \cite{WP}, \cite{ZFC}, \cite{[G]}, \cite{[LXX]}, \cite{R}, \cite{[WXZ]}, \cite{[XLX]}).

\vskip2mm
\noindent
In this paper, by using of the generalized maximum principle, we give a complete classification for  2-dimensional complete  $\lambda$-surfaces in $\mathbb R^3$ with
constant  squared norm  of the second fundamental form. More precisely, we prove the following:
\begin{theorem}\label{theorem 1}
 Let $X: M^2\to \mathbb{R}^{3}$ be a
2-dimensional complete  $\lambda$-surface in $\mathbb R^3$.
If  the squared norm $S$ of the second fundamental form is constant, then
 either $S=0$, or  $S=\dfrac{2+\lambda^2+\lambda\sqrt{\lambda^2+4}}{2}$,  or $S=\dfrac{4+\lambda^2+\lambda\sqrt{\lambda^2+8}}{4}$ and  $X: M^2\to \mathbb{R}^{3}$ is
isometric to one of
\begin{enumerate}
\item $\mathbb R^2$,
\item $S^1(\frac{-\lambda+\sqrt{\lambda^2+4}}{2})\times \mathbb{R}$,
\item $S^2(\frac{-\lambda+\sqrt{\lambda^2+8}}{2})$.
\end{enumerate}

\end{theorem}

\vskip5mm
\section {Preliminaries}
\vskip2mm

\noindent
Let $X: M\rightarrow\mathbb{R}^{n+1}$ be an
$n$-dimensional connected hypersurface of the $n+1$-dimensional Euclidean space
$\mathbb{R}^{n+1}$. We choose a local orthonormal frame field
$\{e_A\}_{A=1}^{n+1}$ in $\mathbb{R}^{n+1}$ with dual coframe field
$\{\omega_A\}_{A=1}^{n+1}$, such that, restricted to $M$,
$e_1,\cdots, e_n$ are tangent to $M^n$.

\noindent
From now on,  we use the following conventions on the ranges of indices:
$$
 1\leq i,j,k,l\leq n
$$
and $\sum_{i}$ means taking  summation from $1$ to $n$ for $i$.
Then we have
\begin{equation*}
dX=\sum_i\limits \omega_i e_i,
\end{equation*}
\begin{equation*}
de_i=\sum_j\limits \omega_{ij}e_j+\omega_{i n+1}e_{n+1},
\end{equation*}
\begin{equation*}
de_{n+1}=\omega_{n+1 i}e_i,
\end{equation*}
where $\omega_{ij}$ is the Levi-Civita connection of the hypersurface.

\noindent By  restricting  these forms to $M$,  we get
\begin{equation}\label{2.1-1}
\omega_{n+1}=0
\end{equation}
and the induced Riemannian metric of the hypersurface  is written as
$ds^2_M=\sum_i\limits\omega^2_i$.
Taking exterior derivatives of \eqref{2.1-1}, we obtain
\begin{equation*}
0=d\omega_{n+1}=\sum_i \omega_{n+1 i}\wedge\omega_i.
\end{equation*}
By Cartan's lemma, we know
\begin{equation*}\label{2.1-2}
\omega_{in+1}=\sum_j h_{ij}\omega_j,\quad
h_{ij}=h_{ji}.
\end{equation*}

$$
h=\sum_{i,j}h_{ij}\omega_i\otimes\omega_j
$$
and
$$
H= \sum_i\limits h_{ii}
$$
are called  the second fundamental form and the mean curvature  of $X: M\rightarrow\mathbb{R}^{n+1}$, respectively.
Let $S=\sum_{i,j}\limits (h_{ij})^2$ be  the squared norm
of the second fundamental form  of $X: M\rightarrow\mathbb{R}^{n+1}$.
The induced structure equations of $M$ are given by
\begin{equation*}
d\omega_{i}=\sum_j \omega_{ij}\wedge\omega_j, \quad  \omega_{ij}=-\omega_{ji},
\end{equation*}
\begin{equation*}
d\omega_{ij}=\sum_k \omega_{ik}\wedge\omega_{kj}-\frac12\sum_{k,l}
R_{ijkl} \omega_{k}\wedge\omega_{l},
\end{equation*}
where $R_{ijkl}$ denotes components of the curvature tensor of the hypersurface.
Hence,
the Gauss equations are given by
\begin{equation}\label{eq:2.1-3}
R_{ijkl}=h_{ik}h_{jl}-h_{il}h_{jk}.
\end{equation}

\noindent
Defining the
covariant derivative of $h_{ij}$ by
\begin{equation}\label{2.1-6}
\sum_{k}h_{ijk}\omega_k=dh_{ij}+\sum_kh_{ik}\omega_{kj}
+\sum_k h_{kj}\omega_{ki},
\end{equation}
we obtain the Codazzi equations
\begin{equation}\label{eq:16-5}
h_{ijk}=h_{ikj}.
\end{equation}
By taking exterior differentiation of \eqref{2.1-6}, and
defining
\begin{equation}\label{eq:16-3}
\sum_lh_{ijkl}\omega_l=dh_{ijk}+\sum_lh_{ljk}\omega_{li}
+\sum_lh_{ilk}\omega_{lj}+\sum_l h_{ijl}\omega_{lk},
\end{equation}
we have the following Ricci identities:
\begin{equation}\label{eq:6-2-6}
h_{ijkl}-h_{ijlk}=\sum_m
h_{mj}R_{mikl}+\sum_m h_{im}R_{mjkl}.
\end{equation}
Defining
\begin{equation}\label{eq:16-6}
\begin{aligned}
\sum_mh_{ijklm}\omega_m&=dh_{ijkl}+\sum_mh_{mjkl}\omega_{mi}
+\sum_mh_{imkl}\omega_{mj}+\sum_mh_{ijml}\omega_{mk}\\
&\ \ +\sum_mh_{ijkm}\omega_{ml}
\end{aligned}
\end{equation}
and taking exterior differentiation of  \eqref{eq:16-3}, we get
\begin{equation}\label{eq:16-2-8}
\begin{aligned}
h_{ijkln}-h_{ijknl}&=\sum_{m} h_{mjk}R_{miln}
+ \sum_{m}h_{imk}R_{mjln}+ \sum_{m}h_{ijm}R_{mkln}.
\end{aligned}
\end{equation}
For a smooth function $f$, we define
\begin{equation}\label{2.1-14}
\sum_i f_{,i}\omega_i=df,
\end{equation}
\begin{equation}\label{2.1-15}
\sum_j f_{,ij}\omega_j=df_{,i}+\sum_j
f_{,j}\omega_{ji},
\end{equation}
\begin{equation}\label{2.1-15}
|\nabla f|^2=\sum_{i }(f_{,i})^2,\ \ \ \  \Delta f =\sum_i f_{,ii}.
\end{equation}
The $\mathcal{L}$-operator is defined by
\begin{equation}
\mathcal{L}f=\Delta f-\langle X,\nabla f\rangle,
\end{equation}
where $\Delta$ and $\nabla$ denote the Laplacian and the gradient
operator, respectively.
\vskip2mm
\noindent
Formulas in the following Lemma 2.1 can be found in  \cite{CW2}.
\begin{lemma}
Let $X:M^n\rightarrow \mathbb{R}^{n+1}$ be an $n$-dimensional $\lambda$-hypersurface in $\mathbb R^{n+1}$. We have
\begin{equation}\label{eq:16-13}
\mathcal{L}H=H
+S(\lambda-H).
\end{equation}
\begin{equation}
\aligned
\frac{1}{2}\mathcal{L}
|X|^{2}=n-|X|^{2}+\lambda (\lambda-H).
\endaligned
\end{equation}
\begin{equation}\label{eq:18-3}
\frac{1}{2}\mathcal{L}S
=\sum_{i,j,k}h_{ijk}^2+(1-S)S+\lambda f_3,
\end{equation}
\begin{equation}
\frac{1}{2}\mathcal{L}H^{2}=|\nabla H|^2+H^2+S(\lambda-H)H.
\end{equation}
\end{lemma}

\noindent
\begin{lemma}
Let $X:M^2\rightarrow \mathbb{R}^{3}$ be  a $2$-dimensional $\lambda$-surface in $\mathbb R^{3}$. If $S$ is constant,  we have
\begin{equation}\label{eq:6-2-17}
\aligned
\frac{1}{2}\mathcal{L}\sum_{i, j,k}(h_{ijk})^{2}
=&\sum_{i,j,k,l}(h_{ijkl})^{2}+(2-S)\sum_{i,j,k}(h_{ijk})^{2}+6\sum_{i,j,k,l,p}h_{ijk}h_{il}h_{jp}h_{klp}\\
&-3\sum_{i,j,k,l,p}h_{ijk}h_{ijl}h_{kp}h_{lp}+3\lambda \sum_{i,j,k,l}h_{ijk}h_{ijl}h_{kl}
\endaligned
\end{equation}

and
\begin{equation}\label{eq:18-6-15-1}
\aligned
&\frac{1}{2}\mathcal{L}\sum_{i,j,k}(h_{ijk})^{2}\\
=&\dfrac{3}{2}\lambda H|\nabla H|^{2}+\frac{3}{4}\lambda H^3+\frac{3}{4}\lambda H^2 S(\lambda-H)
-\dfrac{3}{4}\lambda S H-\dfrac{3}{4}\lambda S^2 (\lambda-H).
\endaligned
\end{equation}
\end{lemma}
\begin{proof} By making use of the Ricci identities  \eqref{eq:6-2-6},  \eqref{eq:16-2-8} and a direct calculation, we
can  obtain \eqref{eq:6-2-17}. From  the formula \eqref{eq:18-3} in  Lemma 2.1 and  $f_3=\frac{H(3S-H^2)}{2}$, we can prove
\eqref{eq:18-6-15-1}.
\end{proof}

 \vskip10mm
\section{Proofs of the main results}

\vskip5mm
\noindent
If $\lambda=0$, we know that $X:M^2\rightarrow \mathbb{R}^{3}$ is a self-shrinker. From Theorem CO in Section one,
we know that our results are proved. Hence, we only consider the case $\lambda\neq 0$ in this section.
\vskip2mm
\noindent
In order to  prove  our results,  the following
generalized maximum principle for $\mathcal{L}$-operator on $\lambda$-hypersurfaces will play an important role,  which was proved
by Cheng, Ogata and Wei in \cite{COW}:
\begin{lemma} {\rm(}Generalized maximum principle for $\mathcal{L}$-operator {\rm)}
Let $X: M^n\to \mathbb{R}^{n+1}$   be an $n$-dimensional  complete $\lambda$-hypersurface  with Ricci
curvature bounded from below. Let $f$ be any $C^2$-function bounded
from above on this $\lambda$-hypersurface. Then, there exists a sequence of points
$\{p_m\}\subset M^n$, such that
\begin{equation*}
\lim_{m\rightarrow\infty} f(X(p_m))=\sup f,\quad
\lim_{m\rightarrow\infty} |\nabla f|(X(p_m))=0,\quad
\limsup_{m\rightarrow\infty}\mathcal{L} f(X(p_m))\leq 0.
\end{equation*}
\end{lemma}
\noindent

\noindent
First of all, we prove the following:

\begin{theorem}\label{theorem 3}
For   a $2$-dimensional complete $\lambda$-surface $X:M^2\rightarrow \mathbb{R}^{3}$  with constant squared norm $S$ of the second fundamental form, we have either
\begin{enumerate}
\item $\lambda^2 S=(S-1)^2$ and $\sup H^2=S$, or
\item $\lambda^2S=2(S-1)^2$ and $\sup H^2=2S$, or
\item
$\lambda^2S=\dfrac{2(1+S)^2}{9}$ and $\sup H^2=2S$.
\end{enumerate}
\end{theorem}

\begin{proof}
From Lemma 2.1, we have
\begin{equation*}
\dfrac12\mathcal{L}H^2=|\nabla  H|^2+H^2+S(\lambda-H)H.
\end{equation*}
At each point $p\in M^2$, we choose $e_1$  and $e_2$ such that
$$
h_{ij}=\lambda_i\delta_{ij}.
$$
From $2ab\leq \epsilon a^2+\dfrac1{\epsilon}b^2$, we obtain
$$
S=\lambda_1^2 +\lambda_2^2, \ \ H^2=(\lambda_1+\lambda_2)^2\leq 2(\lambda_1^2 +\lambda_2^2)=2S.
$$
Hence, we have on $M^2$
$$
H^2\leq 2S
$$
and the equality holds if and only if $\lambda_1=\lambda_2$. Since $S$ is constant, from the Gauss equations, we know that the Ricci curvature of $X: M^2\to \mathbb{R}^{3}$ is bounded from below. We can apply the generalized maximum principle for $\mathcal L$-operator
to the function $H^2$. Thus, there exists a sequence $\{p_m\}$ in $M^2$ such that
\begin{equation*}
\lim_{m\rightarrow\infty} H^2(p_m)=\sup H^2,\quad
\lim_{m\rightarrow\infty} |\nabla H^2(p_m)|=0,\quad
\limsup_{m\rightarrow\infty}\mathcal{L}H^2(p_m)\leq 0.
\end{equation*}
Since $X: M^2\to \mathbb{R}^{3}$ is a $\lambda$-surface, we have
\begin{equation}
H_{,i}=\sum_k h_{ik} \langle X,e_k \rangle, \quad i=1, 2.
\end{equation}
If $\sup H^2=0$, then $H\equiv0$. From the formula
\begin{equation}
\mathcal{L}H=H+S(\lambda-H),
\end{equation}
we get $\lambda S\equiv 0$. We conclude that   $S\equiv 0$ and $X: M^2\to \mathbb{R}^{3}$ is $\mathbb R^2$.  It is impossible because of  $\lambda\neq 0$. Hence,  we have $\sup H^2>0$.  Without loss of the generality, at each point $p_m$, we can assume $H(p_m)\neq 0$. From  \eqref{eq:2.1-3},
\eqref{eq:6-2-6}, \eqref{eq:16-2-8}, Lemma 2.1 and the definition of $S$, we know that
$\{h_{ij}(p_m)\}$,  $\{h_{ijk}(p_m)\}$ and $\{h_{ijkl}(p_m)\}$ are bounded sequences for $ i, j, k, l=1,2$.
We can assume
$$
 \lim_{m\rightarrow\infty}h_{ij}(p_m)=\bar h_{ij}=\bar \lambda_i\delta_{ij}, \quad  \lim_{m\rightarrow\infty}h_{ijk}(p_m)=\bar h_{ijk},
\quad \lim_{m\rightarrow\infty}h_{ijkl}(p_m)=\bar h_{ijkl},
$$
for $i, j, k, l=1, 2$.
\vskip1mm
\noindent
From Lemma 2.1, we  get
\begin{equation}
\begin{cases}
\begin{aligned}
&\lim_{m\rightarrow\infty} H^2(p_m)=\sup H^2=\bar H^2,\quad
\lim_{m\rightarrow\infty} |\nabla H^2(p_m)|=0,\\
&0\geq
\lim_{m\rightarrow\infty} |\nabla H|^2(p_m)+\bar H^2+S(\lambda-\bar H)\bar H.
\end{aligned}
\end{cases}
\end{equation}
From $\lim_{m\rightarrow\infty} |\nabla H^2(p_m)|=0$ and $|\nabla H^2|^2=4\sum_i(HH_{,i})^2$, we have
\begin{equation}
\bar H_{,k}=0,
\end{equation}
that is,
\begin{equation}\label{eq:6-23-1}
\bar h_{11k}+\bar h_{22k}=0, \ \ \text{\rm for} \  k=1, 2.
\end{equation}
From the definition of $\lambda$-surfaces, we obtain
\begin{equation}\label{eq:6-23-3}
\begin{cases}
\begin{aligned}
&\bar H_{,1}=\bar \lambda_1\lim_{m\rightarrow\infty} \langle X,e_1\rangle(p_m)=0,\\
&\bar H_{,2}=\bar \lambda_2 \lim_{m\rightarrow\infty} \langle X,e_2 \rangle(p_m)=0.
\end{aligned}
\end{cases}
\end{equation}
Since $S$ is constant, we know
$$
\sum_{i,j}h_{ij}h_{ijk}=0, \ \ \text{for } \ k=1, 2.
$$
Thus,
$$
\sum_{i,j}\bar h_{ij}\bar h_{ijk}=0, \ \ \text{for } \ k=1, 2,
$$
that is,
\begin{equation}\label{eq:6-23-2}
\bar\lambda_1\bar h_{11k}+\bar\lambda_2\bar h_{22k}=0.
\end{equation}
We next consider three cases.

\vskip2mm
\noindent
{\bf Case 1: $ \bar \lambda_2=0$}.

\noindent Since $\bar H^2\neq 0$, we have $\bar \lambda_1\neq 0$. From \eqref{eq:6-23-1} and \eqref{eq:6-23-2}, we get
\begin{equation*}
\bar h_{11k}=\bar h_{22k}=0,
\end{equation*}
for $k=1, 2$. Therefore, we have
 $\bar h_{ijk}=0$ for $i,j,k=1,2$.
 From \eqref{eq:18-3} in Lemma 2.1, we have
 \begin{equation*}
 0=S(1-S)+\lambda \bar H S
 \end{equation*}
 since $f_3=\frac{H(3S-H^2)}{2}$ and $\bar H^2=S$. Then we obtain
 \begin{equation*}
 \lambda\bar H=S-1, \ \ S=\sup H^2=\bar H^2=\dfrac{(S-1)^2}{\lambda^2 }.
 \end{equation*}

\noindent
{\bf Case 2: $ \bar \lambda_1=0$}.

\noindent Since $\bar H^2\neq 0$, we have $\bar \lambda_2\neq 0$. By making use of  the same arguments  in Case 1, we get
\begin{equation*}
 S=\sup H^2=\bar H^2=\dfrac{(S-1)^2}{\lambda^2}.
 \end{equation*}

\noindent
{\bf Case 3: $ \bar \lambda_1\bar \lambda_2 \neq0$}.

\noindent
From \eqref{eq:6-23-3}, we know
\begin{equation}
\lim_{m\rightarrow\infty} \langle X,e_1 \rangle(p_m)=\lim_{m\rightarrow\infty} \langle X,e_2 \rangle(p_m)=0.
\end{equation}
Since
\begin{equation}\label{eq:6-3-1}
\aligned
\nabla_{k}\nabla_{i}H
 =&\sum_{j}h_{ijk}\langle X,e_{j}\rangle+h_{ik}+\sum_{j}h_{ij}h_{jk}(\lambda-H),
\endaligned
\end{equation}
 we conclude
\begin{equation}
\bar H_{,ik}=\bar h_{ik}+\sum_j\bar h_{ij}\bar h_{jk}(\lambda-\bar H).
\end{equation}
In particular, we have
\begin{equation}\label{eq:6-23-7}
\begin{aligned}
&\bar h_{1111}+\bar h_{2211}=\bar \lambda_1+\bar \lambda_1^2(\lambda-\bar H),\\
&\bar h_{1122}+\bar h_{2222}=\bar \lambda_2+\bar \lambda_2^2(\lambda-\bar H),\\
&\bar h_{1112}+\bar h_{2212}=0,\\
&\bar h_{1121}+\bar h_{2221}=0.
\end{aligned}
\end{equation}
From Ricci identities \eqref{eq:6-2-6}, we obtain
\begin{equation}\label{eq:6-23-8}
\begin{aligned}
&\bar h_{1112}-\bar h_{1121}=0, \\
& \bar h_{2212}-\bar h_{2221}=0,\\
&\bar h_{1212}-\bar h_{1221}=\bar \lambda_1\bar \lambda_2(\bar \lambda_1-\bar \lambda_2).
\end{aligned}
\end{equation}
Since  $S$ is constant, we know
\begin{equation}
\sum_{i,j}\bar h_{ij}\bar h_{ijk}=0,\ \ {\text {for}}\ \ k=1,2
\end{equation}
and
\begin{equation}
\sum_{i,j}\bar h_{ijl}\bar h_{ijk}+\sum_{i,j}\bar h_{ij}\bar h_{ijkl}=0,\ \ {\text {for}}\ \ k, l=1,2.
\end{equation}
Especially,
\begin{equation}\label{eq:6-23-9}
\begin{aligned}
&\bar \lambda_1\bar h_{111}+\bar \lambda_2\bar h_{221}=0,\ \ \bar \lambda_1\bar h_{112}+\bar \lambda_2\bar h_{222}=0, \\
& \bar \lambda_1\bar h_{1111}+\bar \lambda_2\bar h_{2211}=-\bar h_{111}^2-2\bar h_{121}^2-\bar h_{221}^2,\\
&\bar \lambda_1\bar h_{1122}+\bar \lambda_2\bar h_{2222}=-\bar h_{112}^2-2\bar h_{122}^2-\bar h_{222}^2,\\
&\bar \lambda_1\bar h_{1112}+\bar \lambda_2 \bar h_{2212}
=-\bar h_{111}\bar h_{112}-2\bar h_{121}\bar h_{122}-\bar h_{221}\bar h_{222}.
\end{aligned}
\end{equation}
{\bf Subcase 3.1: $ \bar \lambda_1\neq\bar \lambda_2$}.

\noindent From \eqref{eq:6-23-1}, \eqref{eq:6-23-2} and \eqref{eq:18-3} in Lemma 2.1, we have
\begin{equation}
\bar h_{ijk}=0,\ \ \ i,j,k=1,2
\end{equation}
and
\begin{equation}
0= S(1-S)+\lambda \bar H \dfrac{3S-\bar H^2}{2}.
\end{equation}
From \eqref{eq:6-23-7}, \eqref{eq:6-23-8} and \eqref{eq:6-23-9}, we get
\begin{equation}\label{eq:6-23-10}
\begin{aligned}
&\bar h_{1112}=\bar h_{1222}=\bar h_{2221}=\bar h_{2111}=0, \\
&\bar h_{2211}=\dfrac{\bar \lambda_1^2}{\bar \lambda_1-\bar \lambda_2}(1+\bar \lambda_1(\lambda-\bar H)),\\
&\bar h_{1111}=-\dfrac{\bar \lambda_1\bar \lambda_2}{\bar \lambda_1-\bar \lambda_2}(1+\bar \lambda_1(\lambda-\bar H)),\\
&\bar h_{1122}=-\dfrac{\bar \lambda_2^2}{\bar \lambda_1-\bar \lambda_2}(1+\bar \lambda_2(\lambda-\bar H)),\\
&\bar h_{2222}=\dfrac{\bar \lambda_1\bar \lambda_2}{\bar \lambda_1-\bar \lambda_2}(1+\bar \lambda_2(\lambda-\bar H)).
\end{aligned}
\end{equation}
According to \eqref{eq:18-3} in Lemma 2.1:
\begin{equation*}
0=\sum_{i,j,k} h_{ijk}^2+S(1-S)+\lambda H \dfrac{3S-H^2}{2},
\end{equation*}
we have, for $l,m=1,2,$
\begin{equation}\label{eq:6-23-11}
\begin{aligned}
&2\sum_{i,j,k}h_{ijk}h_{ijkl}+\dfrac{3\lambda}{2}(S-H^2)H_{,l}=0, \\
&2\sum_{i,j,k}h_{ijkm}h_{ijkl}+2\sum_{i,j,k}h_{ijk}h_{ijklm}+\dfrac{3\lambda}{2}(S-H^2)H_{,lm}-3\lambda HH_{,m}H_{,l}=0.
\end{aligned}
\end{equation}
Since $\bar h_{ijk}=0$, we get from \eqref{eq:6-23-11} that
\begin{equation}
2\sum_{i,j,k}\bar h_{ijkm}\bar h_{ijkl}+\dfrac{3}{2}\lambda (S-\bar H^2)\bar H_{,lm}=0,\ \ \  \ {\text for}\ l,m=1,2,
\end{equation}
specifically,
\begin{equation*}\label{eq:6-23-?}
\begin{aligned}
&2\bar h_{1111}^2+6\bar h_{1121}^2+6\bar h_{1221}^2+2\bar h_{2221}^2+\dfrac{3}{2}\lambda (S-\bar H^2)(\bar h_{1111}+\bar h_{2211})=0, \\
&2\bar h_{1112}^2+6\bar h_{1122}^2+6\bar h_{1222}^2+2\bar h_{2222}^2+\dfrac{3}{2}\lambda (S-\bar H^2)(\bar h_{1122}+\bar h_{2222})=0.
\end{aligned}
\end{equation*}
Hence,  we have the following equations from \eqref{eq:6-23-10},
\begin{equation}\label{eq:6-23-12}
\begin{aligned}
&2\bar h_{1111}^2+6\bar h_{1221}^2+\dfrac{3}{2}\lambda (S-\bar H^2)(\bar h_{1111}+\bar h_{2211})=0, \\
&6\bar h_{1122}^2+2\bar h_{2222}^2+\dfrac{3}{2}\lambda (S-\bar H^2)(\bar h_{1122}+\bar h_{2222})=0.
\end{aligned}
\end{equation}
From \eqref{eq:6-23-7}, \eqref{eq:6-23-10},  \eqref{eq:6-23-12} and $S-H^2=-2\bar \lambda_1\bar\lambda_2$,
we know
\begin{equation}\label{eq:6-23-13}
\begin{cases}
&(1+\bar \lambda_1(\lambda-\bar H))
\biggl(\dfrac{2}{(\bar \lambda_1-\bar \lambda_2)^2}(1+\bar \lambda_1(\lambda-\bar H))
 (\bar \lambda_2^2+3\bar \lambda_1^2)-3\lambda \bar \lambda_2\biggl)=0, \\[4mm]
&(1+\bar \lambda_2(\lambda-\bar H))\biggl(\dfrac{2}{(\bar \lambda_1-\bar \lambda_2)^2}(1+\bar \lambda_2(\lambda-\bar H))
 (3\bar \lambda_2^2+\bar \lambda_1^2)-3\lambda \bar \lambda_1\biggl)=0.
 \end{cases}
 \end{equation}
Since
\begin{equation}\label{eq:6-23-15}
\lambda-\bar H=\dfrac{S(S-1)}{\bar \lambda_1^3+\bar \lambda_2^3}-\bar \lambda_1-\bar \lambda_2,
\end{equation}
by a direct calculation,  from \eqref{eq:6-23-13} and \eqref{eq:6-23-15}, we get that this subcase does not occur.

\noindent
In fact, from \eqref{eq:6-23-13},
we have
$$(1+\bar \lambda_1(\lambda-\bar H))(1+\bar \lambda_2(\lambda-\bar H))=0
$$ or
$$(1+\bar \lambda_1(\lambda-\bar H))(1+\bar \lambda_2(\lambda-\bar H))\neq 0.
$$

\noindent
If $(1+\bar \lambda_1(\lambda-\bar H))(1+\bar \lambda_2(\lambda-\bar H))=0$, we have
$1+\bar \lambda_1(\lambda-\bar H)=0$ and $1+\bar \lambda_2(\lambda-\bar H)\neq 0$ or
$1+\bar \lambda_2(\lambda-\bar H)=0$ and $1+\bar \lambda_1(\lambda-\bar H)\neq 0$
because of   $\bar \lambda_1\neq \bar\lambda_2$. Hence, we get
\begin{equation}\label{eq:7-6-1}
\begin{cases}
& 1+\bar \lambda_1(\lambda-\bar H)=0,\\
& \dfrac{2}{(\bar \lambda_1-\bar \lambda_2)^2}(1+\bar \lambda_2(\lambda-\bar H))
 (3\bar \lambda_2^2+\bar \lambda_1^2)-3\lambda \bar \lambda_1=0,
 \end{cases}
 \end{equation}
or
\begin{equation}\label{eq:7-6-1-2}
\begin{cases}
& 1+\bar \lambda_2(\lambda-\bar H)=0,\\
& \dfrac{2}{(\bar \lambda_1-\bar \lambda_2)^2}(1+\bar \lambda_1(\lambda-\bar H))
 (\bar \lambda_2^2+3\bar \lambda_1^2)-3\lambda \bar \lambda_2=0.
 \end{cases}
 \end{equation}
If \eqref{eq:7-6-1} holds,  we have
\begin{equation}
\lambda-\bar H=-\dfrac{1}{\bar \lambda_1}, \ \ 1+\bar \lambda_2(\lambda-\bar H)=\dfrac{\bar\lambda_1-\bar\lambda_2}{\bar\lambda_1},
\ \ \lambda \bar \lambda_1=\bar \lambda_1(\bar \lambda_1+\bar \lambda_2)-1.
\end{equation}
From \eqref{eq:7-6-1}, we know
\begin{equation}
\dfrac{2(3\bar \lambda_2^2+\bar \lambda_1^2)}{\bar \lambda_1(\bar \lambda_1-\bar \lambda_2)}
 -3(\bar \lambda_1(\bar \lambda_1+\bar \lambda_2)-1)=0,
 \end{equation}
 that is,
 \begin{equation}\label{eq:7-6-4}
 5\bar \lambda_1^2+6\bar \lambda_2^2+3\bar \lambda_1^2\bar \lambda_2^2-3\bar \lambda_1^4-3\bar\lambda_1\bar\lambda_2=0.
 \end{equation}
On the other hand, from \eqref{eq:6-23-15} and $1+\bar \lambda_1(\lambda-\bar H)=0$, we get
\begin{equation}
\bar \lambda_2(\bar \lambda_2-\bar \lambda_1)(\bar \lambda_2-\bar \lambda_1^2(\bar \lambda_2-\bar \lambda_1))=0,
\end{equation}
namely,
\begin{equation}\label{eq:7-6-2}
\bar \lambda_2-\bar \lambda_1^2(\bar \lambda_2-\bar \lambda_1)=0
\end{equation}
since $\bar \lambda_2(\bar \lambda_2-\bar \lambda_1)\neq0$.
From \eqref{eq:7-6-4} and  \eqref{eq:7-6-2},  we obtain
\begin{equation}
5\bar \lambda_1^2+9\bar \lambda_2^2=0.
\end{equation}
It is impossible. By the same assertion, we know that \eqref{eq:7-6-1-2} does not occur also.
Thus, we must have the following equations from \eqref{eq:6-23-13}
\begin{equation}\label{eq:7-6-5}
\begin{cases}
& \dfrac{2}{(\bar \lambda_1-\bar \lambda_2)^2}(1+\bar \lambda_1(\lambda-\bar H))
 (\bar \lambda_2^2+3\bar \lambda_1^2)-3\lambda \bar \lambda_2=0,\\
& \dfrac{2}{(\bar \lambda_1-\bar \lambda_2)^2}(1+\bar \lambda_2(\lambda-\bar H))
 (3\bar \lambda_2^2+\bar \lambda_1^2)-3\lambda \bar \lambda_1=0.
 \end{cases}
 \end{equation}
From \eqref{eq:7-6-5}, we get
\begin{equation}
\dfrac{\bar \lambda_1}{(1+\bar \lambda_2(\lambda-\bar H))
 (\bar \lambda_1^2+3\bar \lambda_2^2)}=\dfrac{\bar \lambda_2}{(1+\bar \lambda_1(\lambda-\bar H))
 (\bar \lambda_2^2+3\bar \lambda_1^2)}.
 \end{equation}
Therefore, we infer
\begin{equation}\label{eq:7-6-6}
3\bar \lambda_1^2+2\bar \lambda_1\bar \lambda_2+3\bar \lambda_2^2+3(\bar \lambda_1+\bar \lambda_2)(\bar \lambda_1^2+\bar \lambda_2^2)(\lambda-\bar H)=0.
\end{equation}
Substituting
\begin{equation}
\lambda=\dfrac{S(S-1)}{\bar \lambda_1^3+\bar \lambda_2^3}=\dfrac{2S(S-1)}{\bar H(3S-\bar H^2)},
\end{equation}
and $2\bar \lambda_1\bar \lambda_2=\bar H^2-S$
into \eqref{eq:7-6-6}, we get
\begin{equation}
2S+\bar H^2+3S\times\dfrac{2S(S-1)}{3S-\bar H^2}-3\bar H^2 S=0,
\end{equation}
namely,
\begin{equation}
S\bar H^2-\bar H^4-9\bar H^2 S^2+3\bar H^4 S+6S^3=0.
\end{equation}
Therefore, we infer
\begin{equation}
(S-\bar H^2)(6S^2-3\bar H^2 S+\bar H^2)=0.
\end{equation}
This is impossible since $S-\bar H^2=-2\bar \lambda_1\bar \lambda_2\neq0$ and $6S^2-3\bar H^2 S+\bar H^2>0$.

\noindent {\bf Subcase 3.2: $ \bar \lambda_1=\bar \lambda_2$}.
\begin{equation}
\sup  H^2=\bar H^2=\lim_{m\rightarrow\infty}H^2(p_m)=(\bar \lambda_1+\bar \lambda_2)^2=2S.
\end{equation}
From \eqref{eq:18-3} in Lemma 2.1,   one has
\begin{equation}
0= \sum_{i,j,k}\bar h_{ijk}^2+S(1-S)+\lambda \bar H \dfrac{S}{2},
\end{equation}
From \eqref{eq:6-23-1}, \eqref{eq:6-23-7}, \eqref{eq:6-23-8} and \eqref{eq:6-23-9}, we get
\begin{equation}\label{eq:6-23-17}
\begin{aligned}
&\bar h_{1112}=\bar h_{2111},\\
&\bar h_{1122}=\bar h_{2211},\\
& \bar h_{1222}=\bar h_{2221},\\
&\bar h_{2222}=\bar h_{1111},\\
& \bar h_{1112}+\bar h_{1222}=0,\\
&\bar h_{2211}+\bar h_{1111}=\bar \lambda_1+\bar \lambda_1^2(\lambda-\bar H),\\
&\bar h_{111}=-\bar h_{122},\ \bar h_{112}=-\bar h_{222}, \\
&\sum_{i,j,k}\bar h_{ijk}^2=4(\bar h_{111}^2+\bar h_{112}^2)=S(S-1)-\dfrac{\lambda\bar HS}2.
\end{aligned}
\end{equation}
From \eqref{eq:6-23-11}, we have
\begin{equation*}\label{eq:6-23-}
\begin{aligned}
&2\bar h_{111}\bar h_{1111}+6\bar h_{112}\bar h_{1121}+6\bar h_{122}\bar h_{1221}+2\bar h_{222}\bar h_{2221}+
\dfrac{3}{2}\lambda(S-\bar H^2)(\bar h_{111}+\bar h_{221})=0,\\
&2\bar h_{111}\bar h_{1112}+6\bar h_{112}\bar h_{1122}+6\bar h_{122}\bar h_{1222}+2\bar h_{222}\bar h_{2222}+
\dfrac{3}{2}\lambda(S-\bar H^2)(\bar h_{112}+\bar h_{222})=0.
\end{aligned}
\end{equation*}
Thus, from \eqref{eq:6-23-17} and the above equations, we obtain
\begin{equation}\label{eq:6-23-16}
\begin{cases}
&\bar h_{122}(-\bar h_{1111}+3\bar h_{2211})+4\bar h_{211}\bar h_{2111}=0,\\[2mm]
&-4\bar h_{122}\bar h_{2111}+\bar h_{211}(-\bar h_{1111}+3\bar h_{2211})=0.
\end{cases}
\end{equation}
If
\begin{equation*}
(-\bar h_{1111}+3\bar h_{2211})^2+16\bar h_{2111}^2\neq0,
\end{equation*}
 we have the following equations from \eqref{eq:6-23-16} and\eqref{eq:6-23-17}
 \begin{equation*}
 \bar h_{ijk}=0,\ \ i,j,k=1,2
 \end{equation*}
 and
 \begin{equation*}
0=S(1-S)+\lambda \bar H \dfrac{S}{2}, \ \ \ \  2S=\sup H^2=\bar H^2=\dfrac{4(S-1)^2}{\lambda^2}.
\end{equation*}
If
\begin{equation*}
(-\bar h_{1111}+3\bar h_{2211})^2+16\bar h_{2111}^2=0,
\end{equation*}
from  \eqref{eq:6-23-17}, we have
\begin{equation}\label{eq:6-23-22}
\bar h_{2111}=\bar h_{2221}=0,\ \bar h_{1111}=3\bar h_{2211}, \ \bar h_{2211}=\dfrac{1}{4}(\bar \lambda_1+\bar \lambda_1^2(\lambda-\bar H)).
\end{equation}
From Lemma 2.2 and $\lambda_1=\lambda_2$, we have
\begin{equation}\label{eq:6-23-20}
\aligned
&\frac{1}{2}\lim_{m\rightarrow\infty}\mathcal{L}\sum_{i, j,k}(h_{ijk})^{2}(p_m)\\
=&\sum_{i,j,k,l}(\bar h_{ijkl})^{2}+(2-S)\sum_{i,j,k}(\bar h_{ijk})^{2}+6\sum_{i,j,k,l,p}\bar h_{ijk}\bar h_{il}\bar h_{jp}\bar h_{klp}\\
&-3\sum_{i,j,k,l,p}\bar h_{ijk}\bar h_{ijl}\bar h_{kp}\bar h_{lp}+3\lambda \sum_{i,j,k,l}\bar h_{ijk}\bar h_{ijl}\bar h_{kl}\\
=&2 \bar h_{1111}^2+6\bar h_{2211}^2+4(2-S)(\bar h_{111}^2+\bar h_{112}^2)+12S(\bar h_{111}^2+\bar h_{112}^2)\\
&-6S(\bar h_{111}^2+\bar h_{112}^2)+6\lambda \bar H(\bar h_{111 }^2+\bar h_{ 112}^2)\\
=&24\bar h_{2211}^2+(8+2S+6\lambda \bar H)(\bar h_{111 }^2+\bar h_{ 112}^2)\\
\endaligned
\end{equation}
and
\begin{equation}\label{eq:6-23-21}
\aligned
&\frac{1}{2}\lim_{m\rightarrow\infty}\mathcal{L}\sum_{i,j,k}(h_{ijk})^{2}(p_m)\\
=&\frac{3}{4}\lambda \bar H^3+\frac{3}{4}\lambda \bar H^2 S(\lambda-\bar H)
-\dfrac{3}{4}\lambda S \bar H-\dfrac{3}{4}\lambda S^2 (\lambda-\bar H)\\
=&\dfrac{3}{4}\lambda S\bar H+\dfrac{3}{4}\lambda S^2(\lambda-\bar H).
\endaligned
\end{equation}
Hence, we get the following equation from \eqref{eq:6-23-20} and \eqref{eq:6-23-21}
\begin{equation}\label{eq:6-23-23}
24\bar h_{2211}^2  +(8+2S+6\lambda \bar H)(\bar h_{111 }^2+\bar h_{ 112}^2)-\dfrac{3}{4}\lambda S\bar H-\dfrac{3}{4}\lambda S^2(\lambda-\bar H)=0.
\end{equation}
From \eqref{eq:6-23-22}  and $\bar \lambda_1=\bar \lambda_2$, we have
\begin{equation*}
\bar h_{2211}=\dfrac{\bar \lambda_1}{8}(2+H(\lambda-\bar H)).
\end{equation*}
Hence, we get
$$
3(2+H(\lambda-\bar H))^2+2(8+2S+6\lambda \bar H)(2S-2-\lambda\bar H)-12\lambda \bar H-12\lambda S(\lambda-\bar H)=0,
$$
that is, from $2\bar H^2=S$,
\begin{equation*}
\bigl(2+\bar H(\lambda-\bar H)\bigl)\biggl(2+2S+3\lambda \bar H\biggl)=0.
\end{equation*}
Thus,  either $2+\bar H(\lambda-\bar H)=0$ or
$2+2S+3\lambda  \bar H=0$. \\
If $2+\bar H(\lambda-\bar H)=0$, we have
\begin{equation*}
2\lambda^2S =\lambda^2\bar H^2=  4 (1-S)^2,
\end{equation*}
that is,
$$
\lambda^2S=2(S-S)^2, \ \ \sup H^2=2S.
$$
If $2+2S+3\lambda  \bar H=0$,
\begin{equation*}
9\lambda^2\bar H^2=4(1+S)^2,
\end{equation*}
that is,
$$
9\lambda^2S= 2(1+S)^2  \ {\rm and } \ \sup H^2=2S.
$$
We complete this proof of Theorem \ref{theorem 3}.
\end{proof}

\begin{theorem}\label{theorem 4}
Let $X:M^2\rightarrow \mathbb{R}^{3}$ be a $2$-dimensional $\lambda$-surface. If either
 $\lambda^2S=(S-1)^2$, or $\lambda^2S=2(S-1)^2$, or $9\lambda^2S=2(S-1)^2$,  then the mean curvature  $H$ satisfies $H\neq 0$ on  $ M^2$.
\end{theorem}
\begin{proof}
If there exists a point $p \in M^2$ such that $H=0$ at $p$, then we know, at $p$,
$$
H=\lambda_1+\lambda_2=0, \ S=\lambda_1^2+\lambda_2^2=2\lambda_1^2.
$$
Since $S\neq0$, we get $\lambda_1(p)=-\lambda_2(p)\neq0$.
From
$$
H_{,i}=\sum_{k}h_{ik}\langle X, e_k\rangle, \ \ \text{\rm for} \ \ i=1, 2,
$$
we have
\begin{equation}\label{eq:6-23-0}
h_{111}+h_{221}=\lambda_1\langle X, e_1\rangle,\ \  h_{112}+h_{222}=\lambda_2\langle X, e_2\rangle.
\end{equation}
Since  $S$ is constant, we obtain,
\begin{equation}\label{eq:6-23-}
\begin{aligned}
&\sum_{i,j} h_{ij} h_{ijk}=0,\ \ {\text {for}}\ \ k=1,2\\
&\sum_{i,j} h_{ijl} h_{ijk}+\sum_{i,j} h_{ij} h_{ijkl}=0,\ \ {\text {for}}\ \ k, l=1,2.
\end{aligned}
\end{equation}
Hence, we have
\begin{equation}\label{eq:6-23-24}
\begin{aligned}
& h_{111}=h_{221}, \  h_{112}=h_{222},\\
&  \lambda_1 (h_{1111}- h_{2211})=-2h_{111}^2-2 h_{112}^2,\\
&\lambda_1( h_{1122}- h_{2222})= -2h_{111}^2-2 h_{112}^2,\\
&\lambda_1 (h_{1112}-   h_{2212})=-4h_{111} h_{112},\\
&\lambda_1 (h_{1121}-   h_{2221})=-4h_{111} h_{112}.
\end{aligned}
\end{equation}
From \eqref{eq:6-3-1}, we know
\begin{equation*}
 H_{,ik}=\sum_{j}h_{ijk}\langle X, e_j\rangle+h_{ik}+\sum_j h_{ij} h_{jk}\lambda.
\end{equation*}
Thus,  from \eqref{eq:6-23-0} and  $h_{111}=h_{221}, \  h_{112}=h_{222}$, we  get
\begin{equation}\label{eq:6-23-25}
\begin{aligned}
& h_{1111}+h_{2211}=\dfrac{2(h_{111}^2-h_{112}^2)}{\lambda_1}+\lambda_1+ \lambda_1^2\lambda,\\
& h_{1122}+h_{2222}=\dfrac{2(h_{111}^2-h_{112}^2)}{\lambda_1}+\lambda_2+\lambda_2^2\lambda,\\
& h_{1112}+h_{2212}=0,\\
& h_{1121}+h_{2221}=0.
\end{aligned}
\end{equation}
From \eqref{eq:18-3} in Lemma 2.1 and \eqref{eq:6-23-24}, we obtain, at $p$,
\begin{equation*}
0=\sum_{i,j,k} h_{ijk}^2+S(1-S)=4(h_{221}^2+h_{112}^2)+S(1-S),
\end{equation*}
that is,
\begin{equation}\label{eq:6-23-26}
4(h_{221}^2+h_{112}^2)=S(S-1),
\end{equation}
and
\begin{equation}
2\sum_{i,j,k}h_{ijk}h_{ijkl}+\dfrac{3\lambda}{2}SH_{,l}=0,\ \ \ \ {\rm  for}\ l=1,2.
\end{equation}
In particular,
\begin{equation*}
\begin{cases}
&2h_{111}h_{1111}+6h_{112}h_{1121}+6h_{122}h_{1221}+2h_{222}h_{2221}+\dfrac{3}{2}\lambda S(h_{111}+h_{221})=0,\\[2mm]
&2h_{111}h_{1112}+6h_{112}h_{1122}+6h_{122}h_{1222}+2h_{222}h_{2222}+\dfrac{3}{2}\lambda S(h_{112}+h_{222})=0,
\end{cases}
\end{equation*}
that is,
\begin{equation}\label{eq:6-23-27}
\begin{cases}
&h_{111}(h_{1111}+3h_{2211}+\dfrac{3}{2}\lambda S)+h_{112}(3h_{2111}+h_{2221})=0,\\[2mm]
&h_{111}(h_{1112}+3h_{1222})+h_{112}(3h_{1122}+h_{2222}+\dfrac{3}{2}\lambda S)=0.
\end{cases}
\end{equation}
From \eqref{eq:6-23-24} and  \eqref{eq:6-23-25}, we have
\begin{equation}\label{eq:6-23-27-2}
\begin{aligned}
&h_{2211}=\dfrac{2h_{111}^2}{\lambda_1}+\dfrac{\lambda_1+\lambda_1^2\lambda}2,\\
&h_{1111}=-\dfrac{2h_{112}^2}{\lambda_1}+\dfrac{\lambda_1+\lambda_1^2\lambda}2,\\
&h_{1122}=-\dfrac{2h_{112}^2}{\lambda_1}+\dfrac{\lambda_2+\lambda_2^2\lambda}2,\\
&h_{2222}=\dfrac{2h_{111}^2}{\lambda_1}+\dfrac{\lambda_2+\lambda_2^2\lambda}2,\\
&h_{2221}=\dfrac{2h_{111}h_{112}}{\lambda_1},\ \ \ \ h_{2111}=-\dfrac{2h_{111}h_{112}}{\lambda_1},\\
&h_{1112}=-\dfrac{2h_{111}h_{112}}{\lambda_1},\ \ \ \ h_{1222}=\dfrac{2h_{111}h_{112}}{\lambda_1}.
\end{aligned}
\end{equation}
From $h_{111}=h_{221}, \  h_{112}=h_{222}$ and $\lambda_1=-\lambda_2$, at $p$,  we know
$$
\sum_{i,j,k,l,p}h_{ijk}h_{il}h_{jp}h_{klp}=0, \ \  \sum_{i,j,k,l}h_{ijk}h_{ijl}h_{kl}=0,
\sum_{i,j,k,l,p}h_{ijk}h_{ijl}h_{kp}h_{lp}=\dfrac{S}2S(S-1).
$$
From Lemma 2.2, we get
\begin{equation}\label{eq:6-23-28}
\aligned
&\frac{1}{2}\mathcal{L}\sum_{i, j,k}(h_{ijk})^{2}\\
=&\sum_{i,j,k,l}(h_{ijkl})^{2}+(2-S)\sum_{i,j,k}(h_{ijk})^{2}+6\sum_{i,j,k,l,p}h_{ijk}h_{il}h_{jp}h_{klp}\\
&-3\sum_{i,j,k,l,p}h_{ijk}h_{ijl}h_{kp}h_{lp}+3\lambda \sum_{i,j,k,l}h_{ijk}h_{ijl}h_{kl}\\
=&h_{1111}^2+3h_{2111}^2+3h_{2211}^2+h_{2221}^2+h_{2222}^2+3b_{1222}^2+3h_{1122}^2+h_{1112}^2\\
&+\dfrac{(4-5S)S(S-1)}2
\endaligned
\end{equation}
and
\begin{equation}\label{eq:6-23-29}
\aligned
&\frac{1}{2}\mathcal{L}\sum_{i,j,k}(h_{ijk})^{2}\\
=&\dfrac{3}{2}\lambda H|\nabla H|^{2}+\frac{3}{4}\lambda H^3+\frac{3}{4}\lambda H^2 S(\lambda-H)
-\dfrac{3}{4}\lambda S H-\dfrac{3}{4}\lambda S^2 (\lambda-H)\\
=&-\dfrac{3}{4}\lambda^2S^2.
\endaligned
\end{equation}
Hence, we get the following equation from \eqref{eq:6-23-28} and \eqref{eq:6-23-29}
\begin{equation}\label{eq:6-23-30}
\aligned
&h_{1111}^2+3h_{2111}^2+3h_{2211}^2+h_{2221}^2+h_{2222}^2+3b_{1222}^2+3h_{1122}^2+h_{1112}^2\\
&+\dfrac{(4-5S)S(S-1)}2+\dfrac{3}{4}\lambda^2S^2=0.
\endaligned
\end{equation}
We now have to consider four cases.

\noindent {\bf Case 1: $h_{111}=0$, $h_{ 112}=0$}.

\noindent
From \eqref{eq:6-23-26}, we have $S=1$. It is impossible from \eqref{eq:6-23-30}.

\noindent {\bf Case 2: $h_{111}=0$, $h_{112}\neq0$}.

\noindent In this case, we know the following equations from \eqref{eq:6-23-27-2} and \eqref{eq:6-23-26},
\begin{equation}
 h_{2221}=h_{2111}=h_{1112}=h_{1222}=0,\ \
\end{equation}
$$
h_{2211}=\dfrac{\lambda_1}{2}+\dfrac{1}{4}S\lambda, \ \
h_{1111}=\dfrac{S\lambda}4+\dfrac{3}2\lambda _1-  S\lambda_1,
$$

$$
  h_{1122}=-\lambda_1 S+\dfrac{\lambda_1}{2}+\dfrac{1}{4}S\lambda, \ \
  h_{2222}=\dfrac{S\lambda}{4}- \dfrac{\lambda_1}2.
$$
Putting  these equations into \eqref{eq:6-23-30}, we get
\begin{equation}\label{eq:6-21-4}
-2S^2+5\lambda^2S-8\lambda_1\lambda S+6S+8\lambda_1\lambda=0.
\end{equation}
From the second equation in \eqref{eq:6-23-27}, we know
\begin{equation}
5\lambda_1\lambda=3S-1.
\end{equation}
Therefore, we conclude
\begin{equation}\label{eq:6-23-31}
25S\lambda^2=2(3S-1)^2 \  {\rm and} \
 -2S^2+5\lambda^2S-8(\dfrac{3S-1}{5})(S-1)+6S=0.
\end{equation}
We  get
\begin{equation*}
S=3, \ \ \lambda^2=\dfrac{128}{75}
\end{equation*}
because of $S\geq 1$ from \eqref{eq:6-23-26}.
It contradicts   either $S\equiv\dfrac{(S-1)^2}{\lambda^2 }$, or $S\equiv\dfrac{2(S-1)^2}{\lambda^2}$, or $S\equiv\dfrac{9\lambda^2 S^2}{2(1+S)^2}$.

\noindent {\bf Case 3: $h_{112}=0$, $h_{111}\neq0$}.

\noindent
By making use of the same assertion as in the above case, we know that it is impossible.

\noindent {\bf Case 4: $h_{211}\neq0$, $h_{221}\neq0$}.

\noindent Putting \eqref{eq:6-23-27-2}  into \eqref{eq:6-23-27}

\begin{equation}\label{eq:6-23-35}
\begin{cases}
&\dfrac{6}{\lambda_1}(h_{111}^2-h_{ 112}^2)+ 2\lambda_1+\lambda S=0,\\[4mm]
&\dfrac{6}{\lambda_1}(h_{111}^2-h_{ 112}^2)+ 2\lambda_2+ \lambda S=0.
\end{cases}
\end{equation}
Hence, we have $\lambda_1=0$, which contradicts $S\neq0$. Hence, we conclude  that  $H\neq 0$ on $M^2$.

\end{proof}

\vskip2mm
\noindent
\begin{theorem}\label{theorem 2}
Let $X:M^2\rightarrow \mathbb{R}^{3}$ be a $2$-dimensional complete $\lambda$-surface with constant squared norm $S$ of the second fundamental form. Then
either $\lambda^2S=(S-1)^2$ and $ \inf H^2=S$, or  $\lambda^2S=2(S-1)^2$ and $\inf H^2=2S$.
\end{theorem}

\begin{proof}
We apply the generalized maximum principle for $\mathcal L$-operator
to the function $-H^2$. Thus, there exists a sequence $\{p_m\}$ in $M^2$ such that
\begin{equation*}
\lim_{m\rightarrow\infty} H^2(p_m)=\inf H^2=\bar H^2, \ \
\lim_{m\rightarrow\infty} |\nabla H^2(p_m)|=0, \ \
\liminf_{m\rightarrow\infty}\mathcal{L} H^2(p_m)\geq 0,
\end{equation*}
that is,
\begin{equation}\label{eq:6-22-1}
\begin{cases}
\begin{aligned}
&\lim_{m\rightarrow\infty} H^2(p_m)=\sup H^2=\bar H^2,\quad
\lim_{m\rightarrow\infty} |\nabla H^2(p_m)|=0,\\
&0\geq
\lim_{m\rightarrow\infty} |\nabla H|^2(p_m)+\bar H^2+S(\lambda-\bar H)\bar H.
\end{aligned}
\end{cases}
\end{equation}
\vskip2mm
\noindent
By taking the limit and making use of the same assertion as in Theorem 3.2, we can prove
$\inf H^2\neq 0$.
Hence, without loss of the generality, we can assume
$$
\lim_{m\rightarrow\infty}h_{ijl}(p_m)=\bar h_{ijl}, \quad \lim_{m\rightarrow\infty}h_{ij}(p_m)=\bar h_{ij}=\bar \lambda \delta_{ij},
\quad \lim_{m\rightarrow\infty}h_{ijkl}(p_m)=\bar h_{ijkl},
$$
for $i, j, k, l=1, 2$.
\noindent
If  $ \bar \lambda_1=0$ or $ \bar \lambda_2=0$, we have
 \begin{equation}
 \lambda\bar H=S-1, \ \ S=\dfrac{(S-1)^2}{\lambda^2} \ {\rm and} \ \inf H^2=S.
 \end{equation}
\noindent
If $ \bar \lambda_1\bar \lambda_2 \neq0$, by making use of the same assertion as in the proof of the
theorem 3.1, we get $ \bar \lambda_1=\bar \lambda_2$.  In this case,
\begin{equation}
\inf  H^2=(\bar \lambda_1+\bar \lambda_2)^2=2S.
\end{equation}
We have
$$
0\leq 2S-H^2\leq \sup (2S-H^2)=2S-\inf H^2=0.
$$
Namely,  we obtain $H$ is constant.
Hence, we conclude from \eqref{eq:16-13}
$$
H^2=2S \ {\rm and}  \ 2S(1-S)^2=H^2(1-S)^2=\lambda^2S^2.
$$
The proof of Theorem \ref{theorem 2} is finished.
\end{proof}

\vskip3mm
\noindent
{\it Proof of Theorem \ref{theorem 1}}.  If $\lambda=0$, $X: M^2\to \mathbb{R}^{3}$ is a complete self-shrinker. From
Theorem CO of Cheng and Ogata \cite{CO}, we know that our results hold. If $\lambda\neq 0$, from Theorem \ref{theorem 2}, we know that  $\lambda^2S=(S-1)^2$  or  $\lambda^2S=2(S-1)^2$. It is easy to check that  $\lambda^2S=(S-1)^2$  and  $\lambda^2S=2(S-1)^2$ do not  hold  simultaneously.   If $\lambda^2S=(S-1)^2$, we have $\inf H^2=S=\sup H^2$ from Theorem 3.1. Hence,   $H$ is constant. If
$\lambda^2S=(S-1)^2$, we have $\inf H^2=2S=\sup H^2$ from Theorem 3.1.  $H$ is also constant.  Thus, we conclude that $X: M^2\to \mathbb{R}^{3}$ is an isoparametric  surface.  By a classification theorem due to Lawson \cite{L}, $X: M^2\to \mathbb{R}^{3}$  is $S^k(r)\times\mathbb{R}^{2-k}$, $k=1,2$. By a direct calculation, we conclude
$X: M^2\to \mathbb{R}^{3}$ is either $S^1(\frac{-\lambda+\sqrt{\lambda^2+4}}{2})\times \mathbb{R}^{1}$, or  $S^2(\frac{-\lambda+\sqrt{\lambda^2+8}}{2})$.
\begin{flushright}
$\square$
\end{flushright}

\end{document}